\documentclass[12pt]{amsart}
\usepackage{Style}
\usepackage{DBP}
\usepackage{lmodern}
\usepackage{adjustbox}

\newcommand{\kdot}{{{\,\begin{picture}(1,1)(-1,-2)\circle*{2}\end{picture}\,}}}

\begin{document}
\title{Higher Du Bois and Higher Rational Pairs}
\author{Haoming Ning, Brian Nugent}
\date{\today}

\begin{abstract}
We extend the notions of higher Du Bois and higher rational singularities to pairs in the sense of the minimal model program. We extend numerous results to these higher pairs, including Bertini type theorems, stability under finite maps and that m-rational pairs are m-Du Bois. We prove these using a generalized Kov\'acs--Schwede-type injectivity theorem for pairs, the main technical result of this paper.
\end{abstract}

\maketitle

\section{Introduction}

Rational singularities \cite{Art66} and Du Bois singularities \cite{Ste83} are among the most important singularities studied in algebraic geometry. A key reason is that the cohomological behavior of varieties with rational (resp. Du Bois) singularities is very similar to smooth varieties (resp. normal crossing varieties). Of their many applications, these two classes serve as powerful tools to study the singularities that arise naturally in the minimal model program. A principle of the minimal model program is that one should study singularities via pairs \cite{Kol97}, which allows freedom in applications and inductive arguments. Along this line, rational singularities were extended to pairs in \cite{Kol13, ST08} and Du Bois to pairs in \cite{Kov11}. 

Recently, driven by developments in Hodge theoretic methods, there has been substantial interest in studying the higher analogs of rational and Du Bois singularities (see \cite{MOPW21, JKSY21, FL22, MP25, SVV23, Kov25}). 

The goal of this paper is to unify these two generalizations. We propose the notion of higher rational and Du Bois pairs, in the sense of the minimal model program, and show that various results characteristic of rational and Du Bois singularities generalize to pairs. 


A variety has strict-$m$-Du Bois singularities if the natural map
\[
\Omega_X^p \rightarrow \uOmega_X^p
\]
from the sheaf of K\"ahler differentials to the Du Bois complex (cf. \ref{notation-pair}) is a quasi-isomorphism for all $p \le m$. This definition works well if $X$ is lci but is too restrictive in general. To resolve this, multiple candidates for higher Du Bois singularities were studied in \cite{SVV23} and \cite{Kov25}, to isolate contributions due to bad behavior of K\"ahler differentials in general.
We generalize each of these notions, as well as their rational analogues, to the pair case, in Definitions \ref{def-sings} and \ref{def-rat-sings}.

A main result of this paper is a generalized Kov\'acs--Schewede-type injectivity theorem in the sense of \cite{Kov16a} and \cite{Kov16b} for the pair case. The generalization of the injectivity theorem to the higher graded pieces of the Du Bois complex was proven in \cite{MP25} assuming $X$ is lci and in \cite{PSV24} assuming $X$ has isolated singularities, and was proven in \cite{Kov25} in general. We extend these recent results to the pair case.

\begin{mythm}{\ref{inj-thm}}
Let $(X, \Sigma)$ be a pair with pre-$(m-1)$-Du Bois singularities (cf. Definition \ref{def-sings}). Then for each $p \le m$, the natural map $\bbD_X(\uOmega_{X,\Sigma}^p) \rightarrow \bbD_X(h^0(\uOmega_{X,\Sigma}^p))$ induces injections on cohomology sheaves,

$$ h^i( \bbD_X(\ou)) \hookrightarrow h^i (\bbD_X(h^0(\uOmega_{X,\Sigma}^p)) $$
for all $i$, where $\bbD_X(-) = R\HOM_X(-, \omega_X^\kdot)[-n]$ is the shifted Grothendieck duality functor.
\end{mythm}

Note that if $m = 0$ we obtain \cite[Theorem B]{Kov16b} and if $\Sigma = 0$, we obtain \cite[Theorem 1.1]{Kov25}.

We use this to show that a number of results generalize to the pair setting. In particular, we prove a splitting criterion for higher Du Bois pairs and use this to show $m$-rational implies $m$-Du Bois for pairs. This is a generalization of \cite[Theorem B]{SVV23} and \cite[Theorem 1.5]{Kov25}.

\begin{mythm}{\ref{splits}}
Let $(X, \Sigma)$ be a complex pair and assume the natural map $h^0(\uOmega^p_{X, \Sigma}) \rightarrow \ou$ admits a left inverse for all $p \le m$. Then $(X, \Sigma)$ is pre-$m$-Du Bois.
\end{mythm}

\begin{mythm}{\ref{pre-Rat-pre-DB}}
Let $(X,\Sigma)$ be a pre-$m$-rational pair. Assume that $X$ has rational singularities. Then $(X,\Sigma)$ is pre-$m$-Du Bois.
\end{mythm}

The injectivity theorem also implies $m$-Du Bois pairs are stable under finite surjective maps, through a generalized version of the splitting in \cite{kim2025traceduboiscomplex}.

\begin{mythm}{\ref{cor-finite-map}}
    Let $(X, \Sigma)$ be a reduced pair with $X$ normal and let $f: Y \rightarrow X$ be a surjective finite map. If $(Y, f^{-1} \Sigma)$ is pre-$m$-Du Bois then $(X, \Sigma)$ is pre-$m$-Du Bois.
\end{mythm}

We also show that these new notions behave well under taking general hyperplane sections, generalizing \cite[Theorem A]{SVV23}.

\begin{mythm}{\ref{thm-hyperplane}}
If $(X,\Sigma)$ has pre-$m$-Du Bois (resp. pre-$m$-rational) singularities, and if $H\subset X$ is a general hyperplane section, then $(H, \Sigma\cap H)$ also has pre-$m$-Du Bois (resp. pre-$m$-rational) singularities.
\end{mythm}

We use this to show the same holds for the other notions of $m$-rational and $m$-Du Bois, under mild additional assumptions, see Corollary \ref{cor-hyperplane}.

\vspace{.1in}

{\bf Acknowledgment} Both authors are partially supported by NSF grant DMS-2100389. The second author was partially supported by NSF research grant DMS-2301374 and by a grant from the Simons Foundation SFI-MPS-MOV-00006719-07. We would like to thank S\'andor Kov\'acs and Jakub Witaszek for helpful discussions and suggestions related to this project.


\section{Preliminaries}\label{section-prelim}

\begin{definition}\label{def-pair}
A \emph{reduced pair} $(X, \Sigma)$ over $\bbC$ is a reduced scheme $X$ of finite type over $\bbC$ and a reduced, closed subscheme $\Sigma\subseteq X$. We say that $(X, \Sigma)$ is a \emph{pure-dimensional} pair if both $X$ and $\Sigma$ are pure-dimensional.
\end{definition}

\begin{notation}\label{notation-pair}
Let $(X, \Sigma)$ be a reduced pair over $\bbC$. We will denote $n=\dim X$ and if $(X,\Sigma)$ is pure-dimensional, we denote $c=\codim_X(\Sigma)$. Denote
\begin{itemize}
    \item $\uOmega^\kdot_X\in D_{\rm filt}(X)$ the Du Bois complex of $X$;

    \item $\uOmega^\kdot_{X,\Sigma}\in D_{\rm filt}(X)$ the Du Bois complex of the pair $(X,\Sigma)$;

    \item $\uOmega^\kdot_X(\log \Sigma)\in D_{\rm filt}(X)$ the logarithmic Du Bois complex of the pair $(X, \Sigma)$.
\end{itemize}
Denote their associated graded quotient in degree $p$ respectively by
\begin{align*}
\uOmega^p_X = \Gr^p_{\rm filt} \uOmega^\kdot_X[p], \quad
\uOmega^p_{X,\Sigma} = \Gr^p_{\rm filt} \uOmega^\kdot_{X,\Sigma}[p], \quad
\uOmega^p_X(\log \Sigma) = \Gr^p_{\rm filt} \uOmega^\kdot_X(\log \Sigma)[p]. \quad
\end{align*}
See \cite[\S 6]{Kol13} for more detail about the Du Bois complex of a pair and \cite{DB81, Kov25} about the log Du Bois complex. The following uses the notation of \cite[\S 3]{Kov25}. Specifically, their respective $0$-th cohomology sheaf are denoted by

\begin{align*}
\tOmega^p_X = h^0(\uOmega^p_X), \quad
\tOmega^p_{X,\Sigma} = h^0(\uOmega^p_{X,\Sigma}), \quad
\tOmega^p_X(\log \Sigma) = h^0(\uOmega^p_X(\log \Sigma)). \\
\end{align*}
And their associated filtration and cofiltration complex are denoted by

\begin{align*}
\uf^p_X = F^p\uOmega^\kdot_X,& \quad
\uf_p^X = \Cone(\uf^{p+1}_X\to \uOmega^\kdot_X),\\
\uf^p_{X, \Sigma} =  F^p\uOmega^\kdot_{X,\Sigma},& \quad
\uf_p^{X,\Sigma} = \Cone(\uf^{p+1}_{X,\Sigma}\to \uOmega^\kdot_{X,\Sigma}).
\end{align*}
The sheaves $\tOmega^p_X$ and $\tOmega^p_{X, \Sigma}$ form a complex with the ``filtration b\^ete". We denote analogously

\begin{align*}
\tf^p_X = F^p\tOmega^\kdot_X,& \quad
\tf_p^X = \Cone(\tf^{p+1}_X\to \tOmega^\kdot_X),\\
\tf^p_{X, \Sigma} =  F^p\tOmega^\kdot_{X,\Sigma},& \quad
\tf_p^{X,\Sigma} = \Cone(\tf^{p+1}_{X,\Sigma}\to \tOmega^\kdot_{X,\Sigma}).
\end{align*}

We denote by $\bbD_X(-)$ the shifted Grothendieck duality functor $R\HOM(-, \omega_X^\kdot)[-n]$.

We will reserve the notation $\pi: \widetilde{X} \to X$ for a strong log resolution of the pair $(X, \Sigma)$ with reduced exceptional divisor $E\subset \widetilde{X}$, and $\widetilde{\Sigma}=\pi^{-1}_*\Sigma$ the strict transform. We will also use hyperresolution to refer specifically to cubical hyperresolution, see \cite[\S 10.6]{Kol13} for more detail and conventions. In particular, an embedded hyperresolution of $\Sigma\subseteq X$ refers to a hyperresolution of the $1$-cubical variety $\Sigma \to X$. A good hyperresolution of a pair $(X, \Sigma)$ refers to a hyperresolution of $X$ such that on each component $X_\alpha$, we have $\Sigma_\alpha$ is either an snc divisor, all of $X_\alpha$ or empty.
\end{notation}

In the following we show two basic facts that will be used later in the paper.

\begin{lemma} \label{S2-seq-trick}
    Let $X$ be a noetherian scheme and let $\Sigma$ be a reduced divisor on $X$. Let

    $$ 0 \rightarrow \mathscr{F} \rightarrow \mathscr{F}_0 \rightarrow \mathscr{F}_1 $$
be an exact sequence of coherent sheaves on $X$. If $\mathscr{F}_0$ is $S_2$ and torsion free on $X$ and $\mathscr{F}_1$ is torsion free on $\Sigma$ then $\mathscr{F}$ is $S_2$ and torsion free on $X$.
\end{lemma}

\begin{proof}
    The image of $\mathscr{F}_0$ in $\mathscr{F}_1$ is torsion free on $\Sigma$, therefore it is $S_1$ and has dimension $\dim X - 1$ (or it is zero and the statement is trivial). Now the lemma follows from \cite[4.10(iii)]{Kov25}.
\end{proof}

\begin{lemma}\label{lem-rlim}
    Let $D = B \rightarrow A \leftarrow C$ be a diagram of objects of $D(X)$. There is an exact triangle

\begin{equation*}\label{eqn-rlim-triangle}
R\varprojlim D \rightarrow B \oplus C \rightarrow A \xrightarrow{+1} 
\end{equation*}
\end{lemma}

\begin{proof}

Let $D = B \rightarrow A \leftarrow C$ be a diagram of objects of $D(X)$. Let $D' = B \rightarrow 0 \leftarrow C$ and $D'' = 0 \rightarrow A \leftarrow 0$.

We have the following exact triangle of diagrams,

$$ D'' \rightarrow D \rightarrow D' \xrightarrow{+1}. $$
So we get an exact triangle after applying $R\varprojlim$,

\begin{equation} \label{triangle}
    R\varprojlim D'' \rightarrow R\varprojlim D \rightarrow R\varprojlim D' \xrightarrow{+1}.
\end{equation}
Let $I_1$ be an injective resolution of $B$ and $I_2$ an injective resolution of $C$. Then $I_1 \rightarrow 0 \leftarrow I_2$ is an injective resolution of $D'$ and therefore $R\varprojlim D' \cong \varprojlim (I_1 \rightarrow 0 \leftarrow I_2) \cong I_1 \oplus I_2 \cong B \oplus C$.

Consider the short exact sequence of diagrams

\begin{center}
\begin{tikzcd}[cramped, row sep=small]
0 \arrow[r] \arrow[d] & A \arrow[r] \arrow[d] & A \arrow[d] \\
A \arrow[r]           & A \arrow[r]           & 0           \\
0 \arrow[r] \arrow[u] & 0 \arrow[r] \arrow[u] & 0. \arrow[u]
\end{tikzcd}
\end{center}
Taking $R\varprojlim$, we obtain

$$ R\varprojlim(0 \rightarrow A \leftarrow 0) \rightarrow R\varprojlim(A \rightarrow A \leftarrow 0) \rightarrow R\varprojlim (A \rightarrow 0 \leftarrow 0) \xrightarrow{+1} $$

Let $I$ be an injective resolution of $A$. Then $I \rightarrow I \leftarrow 0$ and $I \rightarrow 0 \leftarrow 0$ are injective resolutions of $A \rightarrow A \leftarrow 0$ and $A \rightarrow 0 \leftarrow 0$ respectively. A similar computation shows that $R\varprojlim(A \rightarrow A \leftarrow 0) \cong 0$ and $R\varprojlim (A \rightarrow 0 \leftarrow 0) \cong A$. Thus $R\varprojlim D'' \cong A[1]$.

Now we can plug these into \ref{triangle} and obtain 

\begin{equation*}
    A[1] \rightarrow R\varprojlim D \rightarrow B \oplus C \xrightarrow{+1} 
\end{equation*}
which is exactly what we want after shifting.
\end{proof}

\section{Logarithmic Du Bois Complex of a Pair}


\begin{definition}\label{def_pair_log_poles}
Let $(X, \Sigma)$ be a reduced pair, $Z$ a reduced closed subscheme of $X$. Define $\uOmega^\kdot_{X,\Sigma}(\log Z)\in D_{\rm filt, coh}(X)$ to be the unique filtered object fitting in the following exact triangle
\[
\uOmega^\kdot_{X,\Sigma}(\log Z) \to \uOmega^\kdot_X(\log Z) \to \uOmega^\kdot_{\Sigma}(\log (Z \cap {\Sigma})) \xrightarrow{+1}.
\]
We call $\uOmega^\kdot_{X,\Sigma}(\log Z)$ the \emph{logarithmic Du Bois complex of $(X,\Sigma)$ with poles along $Z$}. This definition makes sense in general, but we will mostly only consider when $Z$ and $\Sigma$ have no common components.
\end{definition}

Consider an embedded hyperresolution of $\Sigma \subseteq X$ such that on each pieces $X_\alpha$, we have that $Z_\alpha=\pi_\alpha^*Z\subseteq X_\alpha$ is either an snc divisor, all of $X_\alpha$, or empty. It is easy to see that such a hyperresolution always exists, by simply requiring this property on each successive resolution of $X$ when constructing an embedded hyperresoltion. Furthermore, if $Z$ and $\Sigma$ has no common components, we can always choose $\Sigma_\kdot$ such that $Z_\alpha$ and $\Sigma_\alpha$ share no common components with  for each $\alpha$.

\begin{lemma}\label{lem-pair-log-poles-hyperres}
Suppose $Z$ and $\Sigma$ have no common irreducible components. Let $\pi_\kdot:X_\kdot \to X$ be a hyperresolution for $X, \Sigma, Z$ as above. 
Denote 
\[
\Omega^\kdot_{X_\alpha,\Sigma_\alpha}(\log Z_\alpha):=\Omega^\kdot_{X_\alpha}(\log(\Sigma_\alpha+Z_\alpha))(-\Sigma_\alpha),
\]
where it is understood to be 
$0$ when $\Sigma_\alpha\cup Z_\alpha=X_\alpha$. Then we have a filtered quasi-isomorphism
\[
R{\pi_\kdot}_*\Omega^\kdot_{X_\kdot,\Sigma_\kdot}(\log Z_\kdot)\qis \uOmega^\kdot_{X,\Sigma}(\log Z).
\]
\end{lemma}

\begin{proof}
On each pieces $\alpha$ of the hyperresolution, $\Sigma_\alpha$ is smooth, and $Z_\alpha|_{\Sigma_\alpha}$ is a divisor on $\Sigma_\alpha$. Therefore \cite[2.3]{EV92} applies and we have short exact sequences:
\[
0 \to \Omega^\kdot_{X_{\alpha}}(\log \Sigma_{\alpha} + Z_{\alpha})(-\Sigma_\alpha) \to \Omega^\kdot_{X_\alpha}(\log Z_\alpha) \to \Omega^\kdot_{\Sigma_\alpha}(\log Z_\alpha|_{\Sigma_\alpha}) \to 0.
\]
Note that this holds regardless of whether $\Sigma_\alpha =X_\alpha$, $Z_\alpha = X_\alpha$, or otherwise. Applying $(R\pi_\kdot)_*$ gives a distinguished triangle
\[
R{\pi_\kdot}_*\Omega^\kdot_{X_\kdot,\Sigma_\kdot}(\log Z_\kdot)\to \uOmega^\kdot_X(\log Z) \to \uOmega^\kdot_{\Sigma}(\log Z|_{\Sigma}) \xrightarrow{+1}.
\]
The desired quasi-isomorphism immediately follows.
\end{proof}

\begin{corollary}
Let $X$ be smooth and $\Sigma \cup Z$ snc such that $\Sigma$ and $Z$ have no common irreducible components. Then
\[
\uOmega^\kdot_{X, \Sigma}(\log Z) \qis \Omega^\kdot_{X}(\log(\Sigma+Z))(-\Sigma).
\]
\end{corollary}

\begin{remark}
We can view $\uOmega^\kdot_{X,\Sigma}(\log Z)$ as a filtered complex describing the triple $(X, \Sigma, Z)$, where we excise both $\Sigma$ and $Z$ but the complement of $\Sigma$ provides contribution to cohomology with compact support. 

\end{remark}


An important use of the log Du Bois complex of a pair is the following comparison Theorem analogous to \cite[3.15]{FL22}. This will allow for an intrinsic definition (\ref{def-rat-sings}) of $m$-rational singularities.

\begin{theorem}\label{thm-irrtl-pushforward}
Notation as in \ref{notation-pair}. Let $S$ denote the singular locus of the pair $(X, \Sigma)$, that is, the locus where $X$ is singular or $\Sigma$ is not snc. Suppose that $\dim S\le d$. Then for all $p\le n-d$,
    \[
    \bbD_X(\uOmega_X^{n-p}(\log \Sigma))\qis R\pi_* \Omega^p_{\widetilde{X},\widetilde{\Sigma}}(\log E),
    \]
where as in \ref{def_pair_log_poles} we define
    \[
    \Omega^p_{\widetilde{X},\widetilde{\Sigma}}(\log E) = \Omega^p_{\widetilde{X}}(\log \widetilde{\Sigma}+E)(-\widetilde{\Sigma})
    \]
\end{theorem}

We will first develop several facts about the log Du Bois complex of a pair.

\begin{lemma}\label{lem-log-partial-hyperres}
    Let $\pi_\kdot: X_\kdot \to X$ be a cubical scheme over $X$ of cohomological descent (cf. \cite{GNPP88}), and $\Sigma_\kdot = X_\kdot \times \Sigma$. Then
    \[
    \uOmega^\kdot_X(\log \Sigma)\qis R(\pi_\kdot)_* \uOmega^\kdot_{X_\kdot}(\log \Sigma_\kdot)
    \]
\end{lemma}

\begin{proof}
    This is a formal consequence of the independence theorem for filtered log complex \cite[Theorem 6.3]{DB81} and properties of cubical descent, see \cite[I.6, V.3.6(5)]{GNPP88}. More concretely, it suffices to prove the statement when $X_\kdot\to X$ is a $2$-resolution
    \[\begin{tikzcd}
    X_{11} \ar[r] \ar[d] & X_{01} \ar[d] \\
    X_{10} \ar[r] & X.
    \end{tikzcd}\]
    Replace $(X_{01}, \Sigma_{01})$ with a good hyperresolution of itself by a $k$-cubical variety, and view the above as a $(k+2)$-cubical variety by taking appropriate pre-images. This allows us to assume that $X_{01}$ is smooth with $\Sigma_{01}$ either a snc divisor, empty, or equal to $X_{01}$. Now, take another good hyperresolution of the $(k+1)$-cubical variety $(X_{11}, \Sigma_{11})\to (X_{10}, \Sigma_{10})$ to obtain a good cubical hyperresolution $\pi'_\kdot: (X'_\kdot, \Sigma'_\kdot) \to (X,\Sigma)$ of the pair $(X, \Sigma)$. The underlying cubical diagram of $X'_\kdot \to X$ (before taking reduction) factor through $X_\kdot \to X$ in which the latter is cofinal. Therefore, the two inverse limits coincides
    \[
    \uOmega^\kdot_X(\log \Sigma)\qis R(\pi'_\kdot)_* \uOmega^\kdot_{X'_\kdot}(\log \Sigma'_\kdot) \qis R(\pi_\kdot)_* \uOmega^\kdot_{X_\kdot}(\log \Sigma_\kdot),
    \]
    and the claim follows by the independence theorem.
\end{proof}

In particular, a cubical partial hyperresolution as defined in \cite[3.2]{Kov25} is of cohomological descent (and note that there is no assumption needed on $\Sigma_\alpha$). This observation allows us to compute $\uOmega^\kdot_X(\log \Sigma)$ from a cubical scheme that is not necessarily smooth. We will use this to show the log version of \cite[Proposition 3.8]{DB81}.

\begin{proposition}\label{lem-irrtl-triangle}
Let $\pi: \widetilde{X}\to X$ be a proper map that is an isomorphism outside a closed subset $S\subseteq X$, and let $E=\pi^{-1}(S)$ with the reduced structure. Then there exists an exact triangle
    \[
    \uOmega^\kdot_{X}(\log \Sigma)\to \uOmega^\kdot_S(\log \Sigma|_S)\oplus R\pi_* \uOmega^\kdot_{\widetilde{X}}(\log \pi^{-1}(\Sigma)) \to R\pi_* \uOmega^\kdot_{E}(\log \pi^{-1}(\Sigma)|_E) \xrightarrow{+1}.
    \]
where the second map is obtained by taking the difference.
\end{proposition}

\begin{proof}
Note that the maps come from the functorial properties of the log Du Bois complex, see \cite[Theorem 1.2.2]{Kov02}. By assumption, we have the partial $2$-resolution diagram
\[\begin{tikzcd}[cramped]
    E \ar[r] \ar[d] & \widetilde{X} \ar[d] \\
    S \ar[r] & X,
\end{tikzcd}\]
which we may use to compute $\uOmega^\kdot_{X}(\log \Sigma)$ by Lemma \ref{lem-log-partial-hyperres}. By definition, $(\pi_\kdot)_*(\scrF_\kdot) = \varprojlim (\pi_\alpha)_*(\scrF_\alpha)$. Therefore, by Grothendieck's theorem, we have
\[
\uOmega^\kdot_{X}(\log \Sigma) \qis R\varprojlim \biggl(\uOmega^\kdot_S(\log \Sigma|_S)\rightarrow
R\pi_* \uOmega^\kdot_{E}(\log \pi^{-1}(\Sigma)|_E) \leftarrow
R\pi_* \uOmega^\kdot_{\widetilde{X}}(\log \pi^{-1}(\Sigma))\biggr).
\]
The proof now follows from Lemma \ref{lem-rlim}.
\end{proof}

\begin{corollary}\label{cor-irrtl-psuhforward}
Let $\pi: \widetilde{X}\to X$ be a proper map that is an isomorphism outside a closed subset $S\subseteq X$, $E=\pi^{-1}(S)$ with the reduced structure. Then
    \[
    \uOmega^p_{X, S}(\log \Sigma)\qis R\pi_* \uOmega^p_{\widetilde{X},E}(\log \pi^{-1}(\Sigma)).
    \]
\end{corollary}

\begin{proof}
The triangle in Lemma \ref{lem-irrtl-triangle} is precisely equivalent to the cones of the two maps
\begin{align*}
    \uOmega^\kdot_{X}(\log \Sigma)\to& \uOmega^\kdot_S(\log \Sigma|_S), \\ 
    R\pi_* \uOmega^\kdot_{\widetilde{X}}(\log \pi^{-1}(\Sigma)) \to& R\pi_* \uOmega^\kdot_{E}(\log \pi^{-1}(\Sigma)|_E)
\end{align*}
being quasi-isomorphic. Therefore, this follows from definition \ref{def_pair_log_poles}.
\end{proof}

\begin{proof}[Proof of Theorem \ref{thm-irrtl-pushforward}]
    
Given $\pi$ as in Notation \ref{notation-pair}, we compute using Corollary \ref{cor-irrtl-psuhforward}
    \begin{align*}
    \bbD_X(\uOmega^{n-p}_{X,S}(\log \Sigma))
    &\qis R\pi_* \bbD_{\widetilde{X}}(\Omega^{n-p}_{\widetilde{X}}(\log \pi^{-1}\Sigma\cup E)(-E)) \\
    &\qis R\pi_* \Omega^p_{\widetilde{X}}(\log \pi^{-1}\Sigma\cup E)(E - (\pi^{-1}\Sigma\cup E)) \\
    &\qis R\pi_* \Omega^p_{\widetilde{X}}(\log \pi^{-1}\Sigma\cup E)(-\widetilde{\Sigma}) \\
    &\qis R\pi_* \Omega^p_{\widetilde{X},\widetilde{\Sigma}}(\log E)
    \end{align*}
    
Dualizing the defining triangle in \ref{def_pair_log_poles} for $(X, S, \Sigma)$, we have
    \[
    \bbD_X(\uOmega^{n-p}_{S}(\log \Sigma|_{S})) \to
    \bbD_X(\uOmega^{n-p}_X(\log \Sigma)) \to
    \bbD_X(\uOmega^{n-p}_{X,S}(\log \Sigma)) \xrightarrow{+1}.
    \]
For $n-p\ge d=\dim S$, we have $\uOmega^{n-p}_S(\log \Sigma|_S)=0$, so 
\[
\bbD_X(\uOmega^{n-p}_X(\log \Sigma)) \qis \bbD_X(\uOmega^{n-p}_{X,S}(\log \Sigma))\qis R\pi_* \Omega^p_{\widetilde{X},\widetilde{\Sigma}}(\log E),
\]
as needed.
\end{proof}


\section{Singularities}

Let $\Omega^p_{X,\Sigma}$ be the kernel of the natural map $\Omega_X^p\to \Omega_{\Sigma}^p / \Tors(\Omega_{\Sigma}^p)$, where $\Tors(\Omega_{\Sigma}^p)$ denotes its torsion subsheaf. By definition we have a short exact sequence,

\begin{equation} \label{pair-omega}
0 \to \Omega^p_{X,\Sigma} \to \Omega_X^p \to \Omega_{\Sigma}^p / \Tors(\Omega_{\Sigma}^p) \to 0.
\end{equation}

We thus have a natural map $\Omega^p_{X,\Sigma}\to \uOmega^p_{X, \Sigma}$ for every $p$. 
Note that when $(X, \Sigma)$ is snc, we have $\Omega^p_{X,\Sigma} = \Omega^p_X(\log \Sigma)(-\Sigma)$.

There is a natural map between the two objects of interest for Du Bois and rational pairs, see also \cite[Lemma 3.12]{FL22}.

\begin{proposition}\label{prop:map}
For all $0\le p\le n$, there exists a natural map $\uOmega^p_{X,\Sigma} \to \bbD_X(\uOmega^{n-p}_{X}(\log \Sigma))$.
\end{proposition}

\begin{proof}
Let $\pi:\widetilde{X} \to X$ be a log resolution as in Notation \ref{notation-pair}. By functoriality of the log Du Bois complex, we have $\uOmega^{n-p}_{X}(\log \Sigma)\to R\pi_* \Omega^{n-p}_{\widetilde{X}}(\log \pi^{-1}(\Sigma))$. Dualizing and apply Grothendieck duality, we obtain
\[
R\pi_*\bbD_{\widetilde{X}}(\Omega^{n-p}_{\widetilde{X}}(\log \pi^{-1}(\Sigma))) \to 
\bbD_X(\uOmega^{n-p}_{X}(\log \Sigma)).
\]
But note that $\bbD_{\widetilde{X}}(\Omega^{n-p}_{\widetilde{X}}(\log \pi^{-1}(\Sigma)))\isom \Omega^p_{\widetilde{X},\pi^{-1}(\Sigma)}$. We may pre-compose with the map $\uOmega^{p}_{X,\Sigma}\to R\pi_* \Omega^p_{\widetilde{X}, \pi^{-1}(\Sigma)}$ from the functoriality of the Du Bois complex of pairs to obtain
\[
\uOmega^p_{X,\Sigma} \to 
R\pi_* \Omega^p_{\widetilde{X},\pi^{-1}(\Sigma)} \to 
R\pi_*\bbD_{\widetilde{X}}(\Omega^{n-p}_{\widetilde{X}}(\log \pi^{-1}(\Sigma))) \to 
\bbD_X(\uOmega^{n-p}_{X}(\log \Sigma)).
\]
By the usual factorization arguments, this map is independent of the chosen resolution.
\end{proof}

We are ready to state the main definitions.

\begin{definition}\label{def-sings}
Let $(X, \Sigma)$ be a reduced pair (see Definition \ref{def-pair}).
\begin{enumerate}    
    \item $(X, \Sigma)$ has \emph{pre-$m$-Du Bois} singularities if $h^i(\uOmega^p_{X,\Sigma})=0$ for all $i \ge 0$ and $0\le p\le m$.

    \item $(X, \Sigma)$ has \emph{weakly $m$-Du Bois} singularities if it is pre $m$-Du Bois, $X$ is seminormal and $\tOmega_{X, \Sigma}^p$ is $S_2$ for $p \le m$.

    \item $(X, \Sigma)$ has \emph{$m$-Du Bois} singularities if it is weakly $m$-Du Bois, $\tOmega_{X, \Sigma}^p$ is reflexive for $p\le m$, and $\codim_X(\Sing(X,\Sigma)) \ge 2m + 1$.

    \item $(X,\Sigma)$ has \emph{strict-$m$-Du Bois} singularities if the natural map $\Omega^p_{X,\Sigma} \to \uOmega^p_{X,\Sigma}$ is a quasi-isomorphism for all $0\le p\le m$.
\end{enumerate}
\end{definition}

\begin{remark}
    Note that the condition that $\ot$ be $S_2$ forces $\Sigma$ to be a divisor when $X$ is seminormal, since for $p = 0$, $\ot$ is the ideal sheaf of $\Sigma$ (see the proof of \ref{prop-defn-recovers}).
\end{remark}

\begin{definition}\label{def-rat-sings}
Let $(X, \Sigma)$ be a pure-dimensional reduced pair.

\begin{enumerate}    
    \item $(X, \Sigma)$ has \emph{pre-$m$-rational} singularities if $h^i(\bbD_X(\uOmega_{X}^{n-p}(\log \Sigma))) = 0$ for all $i \ge 0$ and $0\le p\le m$.
\end{enumerate}

For (2) and (3), assume that $X$ is normal and $\Sigma$ is a divisor.

\begin{enumerate}
    \setcounter{enumi}{1}
    \item $(X,\Sigma)$ has \emph{$m$-rational} singularities if it is pre $m$-rational, $h^0(\bbD_X(\uOmega_{X}^{n-p}(\log \Sigma)))$ is reflexive for $p\le m$, and 
    $\codim_X(\Sing(X,\Sigma)) \ge 2m + 2$.
    
    \item $(X,\Sigma)$ has \emph{strict-$m$-rational} singularities if the natural map $\Omega^p_{X,\Sigma} \to \bbD_X(\uOmega_{X}^{n-p}(\log \Sigma))$ (\ref{prop:map}) is a quasi-isomorphism for all $0\le p\le m$.

\end{enumerate}
\end{definition}






    

\begin{remark}\label{rmk-db-defn-divisor}
    Note that in the case when $X$ is smooth and $\Sigma$ is an snc divisor, $\Omega^p_{X, \Sigma} = \Omega^p_X(\log \Sigma)(-\Sigma)\qis \bbD_X(\Omega^{n-p}_{X}(\log \Sigma))$. The definition for higher rational singularities therefore needs to involve the logarithmic Du Bois complex just to preserve the smooth case.

    By Theorem \ref{thm-irrtl-pushforward}, $\bbD_X(\uOmega^{n-p}_X(\log \Sigma))\qis R\pi_* \Omega^p_{\widetilde{X}, \widetilde{\Sigma}}(\log E)$, for $p\le \codim_X(\Sing(X, \Sigma))$, using the notation in the theorem. Thus in this case we can check the various $m$-rational singularities on  the pair $(X, \Sigma)$ from a log resolution, akin to the original definition of rational singularities.
\end{remark}

These definitions recovers the existing definition for higher Du Bois and rational singularities in the literature. When $p=0$, we can also precisely say how each notion is related.

\begin{proposition}\label{prop-defn-recovers}
    Let $(X,\Sigma)$ be a reduced pair.
    \begin{enumerate}
        \item If $\Sigma=\emptyset$, all of the notions in Definition \ref{def-sings}, \ref{def-rat-sings} agrees with those defined in \cite{Kov25, SVV23}.

        \item The notion of strict-$0$-Du Bois pair agrees with Du Bois pair in the usual sense.

        \item If $X$ is semi-normal, the notion of pre-$0$-Du Bois pair agrees with Du Bois pair.
        
        \item If $\Sigma$ is a divisor, the notions of weakly $0$-Du Bois, $0$-Du Bois, strict-$0$-Du Bois, and Du Bois pair all agree.

        \item The notions of $0$-rational, strict-$0$-rational and rational pair in the sense of \cite[2.80]{Kol13} agree.
        
        \item The notions in (5) are further equivalent to $(X, \Sigma)$ being normal (cf. \cite[3.1]{Kov11}) and pre-$0$-rational.
    \end{enumerate}
\end{proposition}
    
\begin{proof}
    The statement of (1) and (2) follows directly from definition. 
    
    If $X$ is seminormal, $\tOmega_{X,\Sigma}^0$ is equal to the kernel of the map

    $$ \OO_X \rightarrow \eta_* \OO_{\Sigma^{sn}} $$
    where $\eta: \Sigma^{sn} \rightarrow \Sigma$ is the seminormalization. This map factors through $\OO_\Sigma$ so its kernel is just $\mathscr{I}_{\Sigma} = \Omega^0_{X,\Sigma}$. This shows being pre-$0$-Du Bois and Du Bois pair are equivalent when $X$ is semi-normal, which is (3). When $\Sigma$ is a divisor, this gives (4) since $\mathscr{I}_\Sigma$ being S2 is equivalent to $\Sigma$ being a divisor. (Note that the codimension condition trivially holds for $p=0$).

    For (5), first note that $h^0(\bbD_X(\uOmega_X^{n}(\log \Sigma)))$ agrees with $\Omega^0_{X, \Sigma} = \mathscr{I}_\Sigma$ on the snc locus of $(X, \Sigma)$. The latter is always reflexive, thus, as $X$ is normal, they agree when the former is. This shows that $0$-rational and strict-$0$-rational are equivalent. (Again note that the codimension condition is trivial since $X$ is normal).
    
    By Theorem \ref{thm-irrtl-pushforward}, $\bbD_X(\uOmega_X^{n}(\log \Sigma)) \qis R\pi_* \scrO_{\widetilde{X}}(-\widetilde{\Sigma})$ for a strong log resolution $\pi: \widetilde{X}\to X$. Since $\pi$ is a thrifty resolution, we always have $R^i\pi_*\omega_{\widetilde{X}}(\widetilde{\Sigma}) =0$ for all $i>0$ by generalized Grauert-Riemenschneider vanishing \cite[10.38(1)]{Kol13}. 
    Now if $(X,\Sigma)$ is strict-$0$-rational, we have $R\pi_*\scrO_{\widetilde{X}}(-\widetilde{\Sigma})\qis \scrO_X(-\Sigma)$. Therefore $\pi$ is thrifty rational resolution of $(X, \Sigma)$, and hence $(X, \Sigma)$ is a rational pair in the sense of Koll\'ar--Kov\'acs. Conversely, if $(X, \Sigma)$ is a rational pair, then every strong log resolution (which is thrifty) is a rational resolution by \cite[2.86]{Kol13}. Again using Theorem \ref{thm-irrtl-pushforward}, we have
    \[
    \bbD_X(\uOmega_X^{n}(\log \Sigma)) \qis R\pi_* \scrO_{\widetilde{X}}(-\widetilde{\Sigma}) \qis \scrO_X(-\Sigma),
    \]
    which shows (5). 
    
    Lastly, observe that by Theorem \ref{thm-irrtl-pushforward}, $(X, \Sigma)$ being normal is precisely $h^0(\bbD_X(\uOmega_X^{n}(\log \Sigma))) = \scrO_X(-\Sigma)$. This gives (6).
\end{proof}

\begin{proposition}\label{prop-lci-recover}
Let $(X, \Sigma)$ be a reduced pair with $\Sigma$ a divisor. Assume that $X$ and $\Sigma$ are local complete intersections. If $(X, \Sigma)$ has $m$-Du Bois (resp. $m$-rational) singularities then $(X, \Sigma)$ has strict-$m$-Du Bois (resp. strict-$m$-rational) singularities.
\end{proposition}
\begin{proof}
    The case $m=0$ is covered by \ref{prop-defn-recovers}, so let $m\ge 1$. Suppose $X$ and $\Sigma$ are local complete intersections. In either case, we have $\codim_X(\Sing(X)) \ge \codim_X(\Sing(X,\Sigma)) \ge 2m + 1$. Therefore, by \cite[Corollary 3.1]{MV19}, $\Omega_X^p$ is reflexive. Now it follows from Lemma \ref{S2-seq-trick} and \ref{pair-omega} that $\Omega_{X,\Sigma}^p$ is reflexive for $p \le m$. 
    
    Now, when $(X, \Sigma)$ is $m$-Du Bois, this gives $\Omega_{X,\Sigma}^p = \ot$ for $p \le m$ since they agree away from a $\codim \ge 3$ locus. Therefore $(X,\Sigma)$ is strict-$m$-Du Bois. Similarly, when $(X, \Sigma)$ is $m$-rational, we have $\Omega_{X,\Sigma}^p = h^0(\bbD_X(\uOmega^{n-p}_X(\log \Sigma)))$, so $(X, \Sigma)$ is strict-$m$-rational.
\end{proof}

\subsection{Two Out of Three}

In this section, we see how a pair $(X, \Sigma)$ being $m$-Du Bois relates to $X$ and $\Sigma$ being $m$-Du Bois. First, recall the following basic property of Du Bois pairs. 

\begin{proposition} \label{230DB}
    Let $(X,\Sigma)$ be a reduced pair. If two of $X, \Sigma$ and $(X,\Sigma)$ are Du Bois then so is the third. 
\end{proposition}
This easily follows from the map of triangles,

\begin{center}
\begin{tikzcd}
\mathscr{I}_{\Sigma} \arrow[d] \arrow[r]    & \mathscr{O}_X \arrow[d] \arrow[r]  & \mathscr{O}_\Sigma \arrow[d] \arrow[r, "+1"]  & {} \\
{\underline{\Omega}^0_{X,\Sigma}} \arrow[r] & \underline{\Omega}^0_{X} \arrow[r] & \underline{\Omega}^0_{\Sigma} \arrow[r, "+1"] & {}
\end{tikzcd}
\end{center}

The same argument works for the strict-$m$-Du Bois case.

\begin{proposition} \label{23strict}
    Let $(X,\Sigma)$ be a reduced pair. If any of the following two holds, then so is the third:
    \begin{enumerate}
        \item $(X, \Sigma)$ is strict-$m$-Du Bois;
        \item $X$ is strict-$m$-Du Bois;
        \item $\Sigma$ is strict-$m$-Du Bois up to torsion, that is, $\Omega_\Sigma^p/\Tors(\Omega_\Sigma^p)\qis \uOmega^p_\Sigma$ for $p\le m$.
    \end{enumerate}
\end{proposition}

\begin{proof}
The statement follows from the map of triangles,
    
\begin{center}
\begin{tikzcd}[row sep=scriptsize]
{\Omega}^p_{X,\Sigma} \arrow[d] \arrow[r]    & {\Omega}^p_{X} \arrow[d] \arrow[r]  & {\Omega}^p_{\Sigma}/\Tors(\Omega^p_\Sigma) \arrow[d] \arrow[r, "+1"]  & {} \\
{\underline{\Omega}^p_{X,\Sigma}} \arrow[r] & \underline{\Omega}^p_{X} \arrow[r] & \underline{\Omega}^p_{\Sigma} \arrow[r, "+1"] & {}
\end{tikzcd}\end{center}\end{proof}

\begin{remark}
    In particular, if $(X, \Sigma)$ is snc, then it is strict-$m$-Du Bois for all $m$. This example also shows that in general $(X, \Sigma)$ and $X$ both being strict-$m$-Du Bois does not imply $\Sigma$ is strict-$m$-Du Bois, as $\Omega^p_\Sigma$ has nontrivial torsion.
\end{remark}

We also have the fairly obvious partial result for pre-$m$-Du Bois singularities,

\begin{proposition} \label{23mDB}
    Let $(X, \Sigma)$ be a pre-$m$-Du Bois pair. Then $X$ is pre-$m$-Du Bois if and only if $\Sigma$ is pre-$m$-Du Bois.
\end{proposition}

\begin{proof}
This follows directly from the long exact sequence for $i\ge 1$
\begin{center}
\begin{tikzcd}
h^i({{\uOmega}^p_{X,\Sigma}}) \arrow[r] & h^i(\underline{\Omega}^p_{X}) \arrow[r] & h^i(\underline{\Omega}^p_{\Sigma}) \arrow[r] &
h^{i+1}({{\uOmega}^p_{X,\Sigma}}).
\end{tikzcd}
\end{center}
\end{proof}

What we will see in Example \ref{23-counterexample} is that \ref{23strict} does not hold for (non-strict) $m$-Du Bois singularities, even in the weakest form. Specifically, we construct a pair where $X$ and $\Sigma$ are 1-Du Bois but $(X, \Sigma)$ is not even pre-$1$-Du Bois.


\begin{blank} \label{23counterexample}
Suppose that for some $0\le p\le n$, $X$ satisfies $\Omega^p_X\isom h^0(\uOmega^p_X)$ and $h^1(\uOmega^p_X)=0$, and that the natural map $\phi:\Omega^p_\Sigma\to h^0(\uOmega^p_\Sigma)$ is not surjective. The condition on $X$ is easily met if it is smooth, or just strict-$p$-Du Bois. 
Consider the following commutative diagram where the bottom row is exact
\begin{equation}\label{23Diagram}
\begin{tikzcd}
    \Omega^p_X \ar[r] \ar[d, "\isom"] & \Omega^p_\Sigma \ar[d, "\phi"] & & \\
    h^0(\uOmega^p_X) \ar[r, "\psi"] & h^0(\uOmega^p_\Sigma) \ar[r] & h^1(\uOmega^p_{X, \Sigma}) \ar[r] & 0.
\end{tikzcd}
\end{equation}
By assumption on $\phi$, we have that $\psi$ must also not be surjective. This implies that $h^1(\uOmega_{X, \Sigma}^p)\ne 0$. In particular, the pair $(X, \Sigma)$ is not pre-$p$-Du Bois.
\end{blank}

\begin{example}
Notation as in \ref{23counterexample}, let $X=\bbP^2$ and $\Sigma\subseteq X$ a cuspidal cubic. Note that $\uOmega^0_\Sigma\qis \nu_*\scrO_{\bbP^1}$, where $\nu:\bbP^1\to \Sigma$ is the normalization. In particular, $\Sigma$ is pre-$0$-Du Bois and locally $h^0(\uOmega_\Sigma^p)$ is the integral closure of $\scrO_\Sigma$. The hypothesis is therefore met for $p=0$, so $(X, \Sigma)$ is not pre-$0$-Du Bois even though both $X$ and $\Sigma$ are.
\end{example}

Despite this example, we can obtain the pre-$0$-Du Bois statement if we assume $\Sigma$ is semi-normal. But note that in this case all the $0$-Du Bois notions agree on $\Sigma$, see \cite[Proposition 4.5]{Kov25}.

\begin{proposition}
Let $(X,\Sigma)$ be a reduced pair. If $X, \Sigma$ are pre $0$-Du Bois and in addition $\Sigma$ is semi-normal, then $(X, \Sigma)$ is pre-$0$-Du Bois.
\end{proposition}

\begin{proof}
As in \ref{23mDB}, we have the required vanishing except $h^1(\uOmega^p_{X, \Sigma})$. By assumption, $\scrO_\Sigma\isom \scrO_{\Sigma_{\rm sn}} \isom h^0(\uOmega^0_\Sigma)$ where $\Sigma_{\rm sn}$ is the semi-normalization of $\Sigma$. Therefore $\psi:h^0(\uOmega_X^0)\to h^0(\uOmega_\Sigma^0)$ in diagram \ref{23Diagram} is surjective, from which we can conclude that $(X, \Sigma)$ is pre-$0$-Du Bois.
\end{proof}

The following example shows that this fails for the higher Du Bois singularities, even assuming semi-normality.

\begin{example} \label{23-counterexample}
Fix $r, d\ge 1$, and let $V$ be the $d$-uple embedding of $\bbP^r$, and $\Sigma$ the affine cone over $V$, embedded inside $X=\bbA^n$ where $n={r+d \choose d}$. We will show that for $r\ge 2,d\ge 2$, $\Sigma$ is $1$-Du Bois and in addition $\Omega^1_\Sigma\to h^0(\uOmega^1_\Sigma)$ is not surjective. Consequently, by the discussion in  \ref{23counterexample}, we will have a pair $(X, \Sigma)$ that fails to be pre-$1$-Du Bois even though both $X, \Sigma$ are $1$-Du Bois.

Note that $\Sigma$ is normal. By \cite[Corollary 7.3(2)]{SVV23}, for $r\ge 2$, $h^0(\uOmega^1_\Sigma)$ is reflexive since $H^0(\bbP^r, \Omega^1_{\bbP^r})=0$. By \cite[Theorem 3.1]{PS24}, for all $i>0$ we have 
\[
\Gamma(\Sigma, h^i(\uOmega_\Sigma^1))\isom \bigoplus_{m\ge 1} H^i(\bbP^r, \Omega^1_{\bbP^r}\otimes \scrO_{\bbP^r}(md)) \oplus \bigoplus_{m\ge 1} H^i(\bbP^r, \scrO_{\bbP^r}(md))
\]
which vanishes for $d\ge 1$. This shows that $\Sigma$ is weakly $1$-Du Bois for $r\ge2, d\ge 1$. When $r\ge 2$, we have $\codim_\Sigma \Sing \Sigma \ge 3$, so in fact $\Sigma$ is $1$-Du Bois.

On the other hand, for all $d\ge 2$, $\Omega^1_\Sigma$ has nontrivial cotorsion by \cite[Proposition 10]{GR11}. Since $h^0(\uOmega^1_\Sigma)$ is reflexive for $r\ge 2$, we have $h^0(\uOmega^1_\Sigma)\isom \Omega_\Sigma^{[1]}$, so the natural map $\phi:\Omega^1_\Sigma\to h^0(\uOmega^1_\Sigma)$ is not surjective. 
In particular, as in \ref{23counterexample}, for $r\ge 2$, $d\ge 2$ we have $h^1(\uOmega^1_{X,\Sigma})\ne 0$ and therefore $(X, \Sigma)$ is not pre-1-Du Bois.
\end{example}

\begin{corollary}\label{cor-23-lci}
    Let $(X, \Sigma)$ be a reduced pair. Assume $X$ and $\Sigma$ are lci. If $X$ and $\Sigma$ are $m$-Du Bois, then $(X, \Sigma)$ is $m$-Du Bois.
\end{corollary}

\begin{proof}
    For lci varieties, $m$-Du Bois and strict-$m$-Du Bois agree \cite[Corollary 5.6]{SVV23}, so $(X,\Sigma)$ is pre-$m$-Du Bois by \ref{23strict}. The codimension condition is immediate since it is satisfied for both $X$ and $\Sigma$. So we just need to check that $\ot$ is $S_2$ for $p \le m$. This follows from Lemma \ref{S2-seq-trick}.
\end{proof}

In fact, by Proposition \ref{prop-lci-recover} and \cite[5.6]{SVV23}, under the assumptions of Corollary \ref{cor-23-lci} $m$-Du Bois implies strict-$m$-Du Bois for either $(X, \Sigma)$, $X$ or $\Sigma$. Therefore Proposition \ref{23strict} also applies here.


\section{Hyperplane Sections and Cyclic Covers}

In the first part of this section we gather several lemmas on the behavior of the Du Bois complex of a pair under general hyperplane sections and cyclic covers. This builds toward a surjectivity statement \ref{prop_surjecitivty_cover}, which is a key ingredient in the proof of the injectivity theorem. 

The non-pair version of several statements can be found in \cite[\S 5, \S 6]{Kov25}. We will highlight the difference in the pair setting in the proof.

\begin{blank}[Notation for hyperplane sections and cyclic covers]\label{notation-cyclic-cover}

Let $(X,\Sigma)$ be a reduced pair, $\scrL$ a semi-ample line bundle on $X$. First denote $\iota:\Sigma\to X$ the closed immersion, $V=X\setminus \Sigma$, and $\scrL_\Sigma =\iota^*\scrL$. Now let $s\in H^0(X,\scrL^n)$ a general section for some $n\gg 0$, such that its image $s_\Sigma\in H^0(\Sigma,\scrL_\Sigma^n)$ is a non-zero (general) section on each component of $\Sigma$. Denote $H=(s=0)$ the zero locus of $s$ and form
\[
\eta: Y=\Spec \bigoplus_{i=0}^{n-1} \scrL^{-i} \to X,
\]
the cyclic cover corresponding to the section $s$. Denote $H'=(\eta^* H)_{\red}$, $\Sigma' = \eta^{-1}\Sigma$, and $V'=\eta^{-1}V=Y\setminus \Sigma'$. Our setup gives $H\cap \Sigma = (s_\Sigma=0)$ is a general hyperplane in $\Sigma$, for which we will denote by both $\Sigma_H$ and $H_\Sigma$ for convenience. Further, we have
\[
\eta_\Sigma:=\eta|_{\Sigma'}:\Sigma'=\Spec \bigoplus_{i=0}^{n-1} \scrL_\Sigma^{-i} \to \Sigma
\]
is the cyclic cover corresponding to the section $s_\Sigma$. Lastly, we will always let $\pi_\kdot:X_\kdot\to X$ be a good hyperresolution for the reduced pair $(X,\Sigma)$. That is, using the notation $V_\kdot = X_\kdot \times_X V$ and $\Sigma_\kdot = X_\kdot \setminus V_\kdot$, we have either $\Sigma_\alpha$ is a snc divisor in $X_\alpha$ or $\Sigma_\alpha=X_\alpha$ (in which case we will view $\Sigma$ as the trivial divisor) for all $\alpha$. We may take $H$ general such that $H_\kdot = X_\kdot \times_X H\to H$ is a hyperresolution, and $H_\alpha + \Sigma_\alpha$ is an snc divisor in $X_\alpha$ whenever $\Sigma_\alpha\ne X_\alpha$.
\end{blank}

\begin{lemma}\label{lem-pair-hyperplane-triangles}
Using notation \ref{notation-cyclic-cover}, we have the following exact triangles for all $p$
\begin{align}
    &\uOmega^p_{X, \Sigma}\to \uOmega^p_{X,\Sigma}(\log H)\to \uOmega^{p-1}_{H,\Sigma_H} \xrightarrow{+1} \label{lem-pair-hyperplane-omega} \\
    &\uf_p^{X, \Sigma}\to \uf_p^{X,\Sigma}(\log H)\to \uf_{p-1}^{H,\Sigma_H}[-1] \xrightarrow{+1}. \label{lem-pair-hyperplane-uf}
\end{align}
Recall that $\uf_p^{-, -}$ and $\tf_p^{-, -}$ are the cofiltration complex associated to $\uOmega^\kdot_{-, -}$ and $\tOmega^\kdot_{-, -}$ respectively, see \ref{notation-pair}.
\end{lemma}

\begin{proof}
This is the logarithmic version of \cite[Lemma 5.7]{Kov25} involving our newly defined Du Bois complex associated to a triple. We will do this also through a streamlined hyperresolution argument. Let $\pi_\kdot: X_\kdot \to X$ be as in \ref{notation-cyclic-cover}. On each component we have short exact sequences by \cite[2.3]{EV92} twisted with $-\Sigma_\alpha$
\[
0\to \Omega^p_{X_\alpha} (\log \Sigma_\alpha)(-\Sigma_\alpha) \to \Omega^p_{X_\alpha}(\log \Sigma_\alpha + H_\alpha)(-\Sigma_\alpha)\to \Omega^{p-1}_{H_\alpha}(\log \Sigma_\alpha\cap H_\alpha)(-\Sigma_\alpha \cap H_\alpha) \to 0
\]
Note the use of the projection formula for the third term. Again with the understanding of $\Sigma_\alpha$ being the trivial divisor when $\Sigma_\alpha=X_\alpha$, this is by definition
\[
0\to \Omega^p_{X_\alpha, \Sigma_\alpha}\to \Omega^p_{X_\alpha,\Sigma_\alpha}(\log H_\alpha)\to \Omega^{p-1}_{H_\alpha,(\Sigma_H)_\alpha} \to 0.
\]
Applying $(R\pi_\kdot)_*$ and using Lemma \ref{lem-pair-log-poles-hyperres} we obtain the desired triangle \ref{lem-pair-hyperplane-omega}
\[
\uOmega^p_{X, \Sigma}\to \uOmega^p_{X,\Sigma}(\log H)\to \uOmega^{p-1}_{H,\Sigma_H} \xrightarrow{+1}.
\]

The same argument as in \cite[Lemma 5.8]{Kov25} gives the statement \ref{lem-pair-hyperplane-uf} on cofiltrations from \ref{lem-pair-hyperplane-omega} by applying the 9-lemma on a hyperresolution.
\end{proof}

\begin{proposition}\label{prop_pair_cover}
Using notation \ref{notation-cyclic-cover}, we have for all $p$
\begin{align}
\eta_* \uOmega^p_{Y, \Sigma'}\qis& \uOmega^p_{X,\Sigma} \oplus\left( \bigoplus_{i=1}^{n-1}\uOmega^p_{X,\Sigma}(\log H)\otimes \scrL^{-i} \right) \label{prop_pair_cover_uOmega},\\
\eta_* \tOmega^p_{Y, \Sigma'}\qis& \tOmega^p_{X,\Sigma} \oplus\left( \bigoplus_{i=1}^{n-1}\tOmega^p_{X,\Sigma}(\log H)\otimes \scrL^{-i} \right) \label{prop_pair_cover_tOmega},\\
\eta_* \uf_p^{Y, \Sigma'}\qis& \uf_p^{X,\Sigma} \oplus\left( \bigoplus_{i=1}^{n-1}\uf_p^{X,\Sigma}(\log H)\otimes \scrL^{-i} \right) \label{prop_pair_cover_uf},\\
\eta_* \tf_p^{Y, \Sigma'}\qis& \tf_p^{X,\Sigma} \oplus\left( \bigoplus_{i=1}^{n-1}\tf_p^{X,\Sigma}(\log H)\otimes \scrL^{-i} \right) \label{prop_pair_cover_tf},
\end{align}
where $\uf_p^{X,\Sigma}(\log H)\otimes \scrL^{-i}$ and $\tf_p^{X,\Sigma}(\log H)\otimes \scrL^{-i}$ are considered as complexes of connections compatible with the hyperfiltration (cf. \cite[\S 2.E]{Kov25}).
\end{proposition}

\begin{proof}
We apply \cite[Theorem 6.2(iv)]{Kov25} to both covers $Y\to X$ and $\Sigma'\to \Sigma$ to obtain natural decompositions compatible with the morphism $\eta_* \uOmega^p_{Y} \to \eta_* \uOmega^p_{\Sigma'}$:
\begin{align*}
\eta_*\uOmega^p_Y &\qis \uOmega^p_X \oplus\left( \bigoplus_{i=1}^{n-1}\uOmega^p_{X}(\log H)\otimes \scrL^{-i} \right), \\
\eta_*\uOmega^p_{\Sigma'} &\qis \uOmega^p_\Sigma \oplus\left( \bigoplus_{i=1}^{n-1}\uOmega^p_{\Sigma}(\log H_\Sigma)\otimes \scrL_\Sigma^{-i} \right).
\end{align*}
As $\eta$ is finite, $R\eta_*=\eta_*$ is exact. So we have the following exact triangle on $X$
\[
\eta_* \uOmega^p_{Y, \Sigma'} \to \eta_* \uOmega^p_{Y} \to \eta_* \uOmega^p_{\Sigma'} \xrightarrow{+1}
\]
Examining each summand and using Definition \ref{def_pair_log_poles}, we obtain the quasi-isomorphism \ref{prop_pair_cover_uOmega}
\[
\eta_* \uOmega^p_{Y, \Sigma'}\qis \uOmega^p_{X,\Sigma} \oplus\left( \bigoplus_{i=1}^{n-1}\uOmega^p_{X,\Sigma}(\log H)\otimes \scrL^{-i} \right).
\]
Note the implicit use of the derived projection formula above due to omitted $\iota_*$ for complex supported on $\Sigma$. Now take the $h^0$ cohomology of \ref{prop_pair_cover_uOmega}, which commutes with exact functors $\eta_*$ and $-\otimes \scrL^j$. This gives \ref{prop_pair_cover_tOmega}.

Now consider the triangle for the hyperfiltered complexes
\[\begin{tikzcd}
\eta_* \uOmega^\kdot_{Y, \Sigma'} \ar[r] & 
\eta_* \uOmega^\kdot_{Y} \ar[r] & 
\eta_* \uOmega^\kdot_{\Sigma'} \ar[r, "+1"] & \,
\end{tikzcd}\]
Using \ref{prop_pair_cover_uOmega} and \cite[Lemma 2.28]{Kov25}, for each $0< i< n$, there exists compatible hyperfiltered complexes $\uOmega^\kdot_{X,\Sigma}(\log H)\otimes \scrL^{-i}$ with connection. Taking $h^0$ and applying \ref{prop_pair_cover_tOmega}, we obtain compatible hyperfiltered complexes and a filtered morphism
\[
\tOmega^\kdot_{X,\Sigma}(\log H)\otimes \scrL^{-i} \to \uOmega^\kdot_{X,\Sigma}(\log H)\otimes \scrL^{-i}.
\]
This again uses the fact that $h^0$ commutes with exact $\eta_*$. Consequently, for each $p$, the cofiltration complexes also admits compatible connection
\[
\tf_p^{X,\Sigma}(\log H)\otimes \scrL^{-i} \to
\uf_p^{X,\Sigma}(\log H)\otimes \scrL^{-i}.
\]

Now the existence is established, the same argument as in \cite[Corollary 6.6]{Kov25} gives the statement on cofiltrations, where \ref{prop_pair_cover_uf} follows from \ref{prop_pair_cover_uOmega} and \ref{prop_pair_cover_tf} follows from \ref{prop_pair_cover_tOmega}.
\end{proof}

\begin{remark}
Note well that the notation $\uf_p^{X,\Sigma}(\log H)\otimes \scrL^{-i}$ and $\tf_p^{X,\Sigma}(\log H)\otimes \scrL^{-i}$ are not considered as tensor products, but simply as complexes of connections with respect to $\scrL$, i.e. the associated graded complexes are $\uOmega_p^{X,\Sigma}(\log H)\otimes \scrL^{-i}$ and $\tOmega_p^{X,\Sigma}(\log H)\otimes \scrL^{-i}$ respectively. Their existence ultimately comes only from the cyclic cover, cf. \cite[3.16]{EV92}.
\end{remark}

\begin{proposition}\label{prop_surjecitivty_cover}
Suppose in addition to \ref{notation-cyclic-cover} that $X$ is a proper, connected variety. Then for each $q$ and $0<i<n$, the natural morphism
\[
\bbH^q(X,\tf_p^{X,\Sigma}(\log H)\otimes \scrL^{-i})\twoheadrightarrow \bbH^q(X,\uf_p^{X,\Sigma}(\log H)\otimes \scrL^{-i})
\]
is surjective.
\end{proposition}
\begin{proof}
By \cite[Corollary 7.5]{Kov25} applied to $(Y, \Sigma')$ we have a surjection
\[
\bbH^q(X,\eta_*\tf_p^{Y,\Sigma'})\isom \bbH^q(Y,\tf_p^{Y,\Sigma'})\twoheadrightarrow \bbH^q(Y,\uf_p^{Y,\Sigma'}) \isom \bbH^q(X,\eta_*\uf_p^{Y,\Sigma'}).
\]
We now simply apply the decompositions \ref{prop_pair_cover_uf} and \ref{prop_pair_cover_tf} to get the surjection on each summand.
\end{proof}

\subsection{Bertini Theorems}

We first prove a conormal sequence for the Du Bois complex of pairs. For simplicity, we suppressed the left derived tensor symbol and denote, for exaxmple, $\uOmega_H^p(-H) =\uOmega_H^p\otimes_{\scrO_H}^L \scrO_H(-H)$ and $\uOmega_X^p|_H = \uOmega_X^p\otimes_{\scrO_X}^L \scrO_H$.

\begin{lemma}\label{lem-conormal}
Let $H\subset X$ be a general hyperplane. Then there is an exact triangle
\[
\uOmega^{p-1}_{H,\Sigma\cap H}(-H) \to \uOmega^p_{X, \Sigma}|_H \to \uOmega^p_{H,\Sigma\cap H} \xrightarrow{+1}
\]
\end{lemma}

\begin{proof}
First, we take $H$ so that $\Sigma \cap H$ is also a general hyperplane section. Consider the following diagram
\begin{equation*}
\begin{tikzcd}[cramped, row sep=scriptsize]
\uOmega^{p-1}_{H,\Sigma\cap H}(-H) \ar[r] \ar[d] & 
\uOmega^{p-1}_H(-H) \ar[r] \ar[d] & 
\uOmega^{p-1}_{\Sigma\cap H}(-\Sigma\cap H) \ar[r,"+1"] \ar[d] & {} \\
\uOmega^{p}_{X,\Sigma}|_H \ar[r] \ar[d] & 
\uOmega^{p}_{X}|_H \ar[r] \ar[d] & 
\uOmega^{p}_{\Sigma}|_{\Sigma\cap H} \ar[r,"+1"] \ar[d] & {} \\
\uOmega^{p}_{H,\Sigma\cap H} \ar[r] \ar[d] & 
\uOmega^{p}_H \ar[r] \ar[d,"+1"] & 
\uOmega^{p-1}_{\Sigma\cap H} \ar[r,"+1"] \ar[d,"+1"] & {} \\
{} & {} & {} & , \\
\end{tikzcd}\end{equation*}
where all pullback and twists are left-derived and we omitted the exact pushforward on the right most column. We note in particular that $\uOmega^{p-1}_{\Sigma\cap H}(-\Sigma\cap H) \qis \uOmega^{p-1}_{\Sigma_\cap H}(-H)$ by the derived projection formula, where the former is twisted before pushforward and the latter after.

Now, each row is exact and the last two columns are exact by \cite[Lemma 3.2]{SVV23}. The first column is then exact by a version of the derived $9$-lemma, see \cite[B.1]{Kov13}.
\end{proof}

\begin{theorem}\label{thm-hyperplane}
Let $H\subset X$ be a general hyperplane section. If $(X,\Sigma)$ has pre-$m$-Du Bois (resp. pre-$m$-rational) singularities, then $(H, \Sigma\cap H)$ also has pre-$m$-Du Bois (resp. pre-$m$-rational) singularities.
\end{theorem}

\begin{proof}
The Du Bois case follows from taking cohomology of the sequence in Lemma \ref{lem-conormal}, as below. By induction on $p$ up to $m$, we may assume that for all $i>0$, $h^i(\uOmega^{p-1}_{H,\Sigma\cap H}(-H)) = 0$. We can choose $H$ general such that taking cohomology commutes with pullback for finitely many given complexes, which is true when $H$ doesn't contain any associated point of their cohomology sheaves. That is, for all $i>0$,
\[
h^i(\uOmega^p_{X, \Sigma}|_H) = h^i(\uOmega^p_{X, \Sigma})|_H = 0.
\] 
Taking cohomology of the conormal sequence \ref{lem-conormal}, we see that $h^i(\uOmega^p_{H,\Sigma\cap H}) =0 $ for all $i>0$ and $0\le p\le m$. Therefore $(H, \Sigma \cap H)$ is pre-$m$-Du Bois.

For the rational case, note that by \cite[Tag 0AU3]{stacks-project}
\[
\bbD_H(-|_H) = i^!\bbD_X(-)[1] = \bbD_X(-)(H)|_H,
\]
where $i:H\to X$ is the inclusion of an effective Cartier divisor. Keep in mind that our convention of $\bbD_X$ includes a shift of $-n$, which produced the $1$ shift in the first quasi-isomorphism. We now dualize the exact triangle of \cite[Lemma 5.5]{Kov25} on $H$ to obtain
\[
\bbD_H(\uOmega_{H}^{(n-1)-(p-1)}(\log \Sigma\cap H)) \to \bbD_X(\uOmega_{X}^{n-p}(\log \Sigma))(H)|_H \to \bbD_H(\uOmega_{H}^{(n-1)-p}(\log \Sigma\cap H))(H)\xrightarrow{+1}.
\]
Again using induction and the fact that $H$ is general, we see that 
\[
h^i(\bbD_H(\uOmega_{H}^{(n-1)-p}(\log \Sigma\cap H))(H)) = 0
\]
for all $i>0$ and $0\le p\le m$. Therefore $(H,\Sigma\cap H)$ is pre-$m$-rational.
\end{proof}

\begin{corollary} \label{cor-hyperplane}
    Let $H\subset X$ be a general hyperplane section, and assume in addition that $X$ has rational singularities. If $(X, \Sigma)$ has $m$-Du Bois (resp. $m$-rational) singularities, then $(H, \Sigma\cap H)$ has $m$-Du Bois (resp. $m$-rational) singularities.
\end{corollary}
\begin{proof}
    Note that either of these assumptions forces $\Sigma$ to be a divisor. Since $H$ is general, we may assume $\Sing(H, \Sigma\cap H)\subseteq \Sing (X, \Sigma)$, which shows that the codimension condition for either situation holds on $H$. Moreover, since $X$ has rational singularities, so does $H$ general. Using lemmas \ref{Rat-Div-h0} and \ref{Rat-Div-h0-intrinsic}, we see that both $\tOmega_{H,\Sigma}^p$ and $h^0(\bbD_X(\uOmega_{H,\Sigma}^{n-p}))$ are reflexive, which is all we needed to show given Theorem \ref{thm-hyperplane}.
\end{proof}

We end the section on Bertini's theorem by deducing the following lemma, which will be useful in establishing the main injectivity theroem.

\begin{lemma}\label{lem-premDB-hyperplane}
If $(X,\Sigma)$ has pre-$m$-Du Bois singularities, we have for all $p\le m$
\begin{equation}
    \tf_p^{X, \Sigma}\to \tf_p^{X,\Sigma}(\log H)\to \tf_{p-1}^{H,\Sigma_H}[-1] \xrightarrow{+1}. \label{lem-pair-hyperplane-tf}
\end{equation}
Consequently, $\tf_p^{X,\Sigma}(\log H)\qis \uf_p^{X,\Sigma}(\log H)$ for $p\le m$.
\end{lemma}

\begin{proof}
By Theorem \ref{thm-hyperplane} $H$ also has pre-$m$-Du Bois singularities. Since $h^1(\uOmega^p_{X, \Sigma})=0$ for all $p\le m$, the long exact sequence on cohomology for \ref{lem-pair-hyperplane-omega} becomes
\[
0\to \tOmega^p_{X, \Sigma}\to \tOmega^p_{X,\Sigma}(\log H)\to \tOmega^{p-1}_{H,\Sigma_H} \to 0.
\]
As $\tf^{X,\Sigma}_p$ and $\tf^{H,\Sigma_H}_p$ are associated to the ``filtration b\^ete'' on $\tOmega_{X,\Sigma}^\kdot$ and $\tOmega_{H,\Sigma_H}^\kdot$ respectively, we automatically have \ref{lem-pair-hyperplane-tf} on cofiltrations.

The final statement then following from the $5$-lemma, \ref{lem-pair-hyperplane-tf}, and \ref{lem-pair-hyperplane-omega}.
\end{proof}


\section{Injectivity Theorem and Applications}

\begin{theorem} \label{inj-thm}
Let $(X, \Sigma)$ be a pair with pre-$(m-1)$-Du Bois singularities. Then for each $p \le m$, the natural map $\bbD_X(\uOmega_{X,\Sigma}^p) \rightarrow \bbD_X(\tOmega_{X,\Sigma}^p)$ induces injections on cohomology sheaves,

$$ h^i( \bbD_X(\ou)) \hookrightarrow h^i (\bbD_X(\ot)) $$
for all $i$.
\end{theorem}

We defer the proof of \ref{inj-thm} until section \ref{proof-inj-thm}.

\begin{theorem} \label{splits}
Let $(X, \Sigma)$ be a complex pair and assume the natural map $\ot \rightarrow \ou$ admits a left inverse for all $p \le m$. Then $(X, \Sigma)$ is pre-$m$-Du Bois.
\end{theorem}

\begin{proof}
By induction, we may assume $X$ is pre-$(m-1)$-Du Bois. Our assumption implies that $\bbD_X(\uOmega_{X,\Sigma}^m) \rightarrow \bbD_X(\tOmega_{X,\Sigma}^m)$ is has a right inverse, so $h^i(\bbD_X(\uOmega_{X,\Sigma}^m)) \rightarrow h^i (\bbD_X(\tOmega_{X,\Sigma}^m))$ is surjective for all $i$. By \ref{inj-thm}, $h^i(\bbD_X(\uOmega_{X,\Sigma}^m)) \rightarrow h^i (\bbD_X(\tOmega_{X,\Sigma}^m))$ is an isomorphism for all $i$, i.e. $\bbD_X(\uOmega_{X,\Sigma}^m) \rightarrow \bbD_X(\tOmega_{X,\Sigma}^m)$ is a quasi-isomorphism. Dualizing again, we get that $\ot \rightarrow \ou$ is a quasi-isomorphism, as desired.
\end{proof}

Our next result is a generalization of \cite[Theorem B]{SVV23} and \cite[10.11]{Kov25} to the pair case. Note that in the non-pair case, when $X$ is normal, pre $m$-rational implies $X$ is rational.

\begin{theorem} \label{pre-Rat-pre-DB}
Let $(X,\Sigma)$ be a pre-$m$-rational pair. Assume that $X$ has rational singularities. Then $(X,\Sigma)$ is pre-$m$-Du Bois.
\end{theorem}

Now fix a strong log resolution $(\widetilde{X}, \widetilde{\Sigma}) \rightarrow (X, \Sigma)$. We need the following result of Kebekus-Schnell on extending logarithmic forms,

\begin{theorem}[{\cite[1.4]{KS21}}] \label{KS}
Let $X$ be a variety with rational singularities. Then
$$ \otw \cong \pi_* \Omega_{\widetilde{X}}^p(\textup{log } E) \cong \Omega_X^{[p]} $$
where $\Omega_X^{[p]}$ is the reflexive hull of $\Omega_X^p$.
\end{theorem}

\begin{lemma} \label{Rat-Div-h0}
Let $(X,\Sigma)$ be a pair. Assume that $X$ is rational and $\Sigma$ is a divisor. Then $\ot$ is reflexive and the natural map

$$ \ot \rightarrow \pi_* \Omega^p_{\widetilde{X},\widetilde{\Sigma}}(\textup{log } E) $$
is an isomorphism for all $p$.
\end{lemma}

\begin{proof}
As both sides agree on the locus where $(X,\Sigma)$ is snc, it suffices to show that both sheaves are $S_2$. First, we show $\ot$ is $S_2$. We have an exact sequence
$$ 0 \rightarrow \ot \rightarrow \widetilde{\Omega}_X^p \rightarrow \widetilde{\Omega}_\Sigma^p $$

Since $X$ is rational, $\widetilde{\Omega}_X^p$ is $S_2$ by \ref{KS}. The sheaf $\widetilde{\Omega}_\Sigma^p$ is torsion free on $\Sigma$ so $\ot$ is $S_2$ by Lemma \ref{S2-seq-trick}. Now we show $\pi_* \Omega^p_{\widetilde{X},\widetilde{\Sigma}}(\textup{log } E)$ is $S_2$ with the same technique. We have an exact sequence,

$$ 0 \rightarrow \pi_* \Omega^p_{\widetilde{X},\widetilde{\Sigma}}(\textup{log } E) \rightarrow \pi_* \Omega^p_{\widetilde{X}}(\textup{log } E) \rightarrow \pi_* \Omega^p_{\widetilde{\Sigma}}(\textup{log } E \cap \widetilde{\Sigma}) $$

Once the again the middle term is $S_2$ by \ref{KS} and the right term is torsion free on $\Sigma$ since it is the pushforward of a torsion free sheaf on $\widetilde{\Sigma}$.
\end{proof}

\begin{lemma} \label{Rat-Div-h0-intrinsic}
    If $X$ is Cohen-Macaulay and normal, then the natural map
    \[
    h^0(\bbD_X(\uOmega^{n-p}_{X}(\log \Sigma)))\to \pi_*\Omega^p_{\widetilde{X}, \widetilde{\Sigma}}(\log E)
    \]
    is an isomorphism for all $p$.
\end{lemma}

\begin{proof}
    By Theorem \ref{thm-irrtl-pushforward} and the defining triangle of the log complex of a pair, we have
    \[
    R\pi_*\Omega^p_{\widetilde{X}, \widetilde{\Sigma}}(\log E) \to
    \bbD_X(\uOmega^{n-p}_X(\log \Sigma)) \to
    \bbD_X(\uOmega^{n-p}_{S}(\log \Sigma|_{S})) \xrightarrow{+1}.
    \]
    Since $X$ is Cohen-Macaulay, we have
    \[
    h^i(\bbD_X(\uOmega^{n-p}_{S}(\log \Sigma|_{S})))=0
    \]
    for $i< c=\codim_X S$. But the assumption implies that $c\ge 2$, since the non-snc locus of $\Sigma$ is automatically of $\codim \ge 2$ in $X$. The isomorphism on $h^0$ then follows immediately from the long exact sequence on cohomology.
\end{proof}

\begin{proof}[Proof of Theorem \ref{pre-Rat-pre-DB}]
By induction, we assume $X$ is pre-$(m-1)$-Du Bois. By Lemma \ref{Rat-Div-h0}, Lemma \ref{Rat-Div-h0-intrinsic} and the assumption that $(X,\Sigma)$ is pre-$k$-rational, we have that the composition

$$ \ot \rightarrow \ou \rightarrow \bbD_X(\uOmega^{n-p}_{X}(\log \Sigma)) $$
is a quasi-isomorphism. Therefore $\ot \rightarrow \ou$ splits and the theorem follows from Theorem \ref{splits}.
\end{proof}

\begin{corollary}
    Let $(X,\Sigma)$ be a $m$-rational (resp. strict-$m$-rational) pair. Assume that $X$ has rational singularities. Then $(X,\Sigma)$ is $m$-Du Bois (resp. strict-$m$-Du Bois). \hfill\qedsymbol
\end{corollary}

Next we prove that being pre-$m$-Du Bois pair decends under finite maps. First we note the following consequence of \cite[Theorem 1.1]{kim2025traceduboiscomplex}.

\begin{theorem} \label{G-splitting}
    Let $(Y, \Delta)$ be a reduced pair. Assume $Y$ and $\Delta$ have compatible $G$ actions where $G$ is a finite group. Let $Y = X / G$ and $\Sigma = \Delta / G$. The composition of morphisms

    $$ \uOmega_{X,\Sigma}^\kdot \rightarrow Rf_* \uOmega_{Y,\Delta}^\kdot \rightarrow R\Gamma^G Rf_* \uOmega_{Y,\Delta}^\kdot $$
    is an isomorphism.
\end{theorem}

\begin{proof}
    This follows directly from applying \cite[Theorem 1.1]{kim2025traceduboiscomplex} to the right two columns of the diagram

\begin{equation*}\begin{tikzcd}[cramped, column sep=scriptsize]
\uOmega_{X,\Sigma}^\kdot \ar[r] \ar[d] & 
\uOmega_X^\kdot \ar[r] \ar[d] & 
\uOmega_\Sigma^\kdot \ar[r,"+1"] \ar[d] & {} \\
Rf_* \uOmega_{Y,\Delta}^\kdot \ar[r] \ar[d] & 
Rf_* \uOmega_{Y}^\kdot \ar[r] \ar[d] & 
Rf_* \uOmega_{\Delta}^\kdot \ar[r,"+1"] \ar[d] & {} \\
R\Gamma^G Rf_* \uOmega_{Y,\Delta}^\kdot \ar[r] & 
R\Gamma^G Rf_* \uOmega_{Y}^\kdot \ar[r] & 
R\Gamma^G Rf_* \uOmega_{\Delta}^\kdot \ar[r,"+1"] & {} \\
\end{tikzcd}\end{equation*}
\end{proof}

\begin{theorem} \label{normal-splitting}
    Let $(X, \Sigma)$ be a reduced pair with $X$ normal and let $f: Y \rightarrow X$ be a surjective finite map. The morphism

    $$ \uOmega_{X,\Sigma}^\kdot \rightarrow Rf_* \uOmega_{Y,f^{-1}\Sigma}^\kdot $$
splits.
\end{theorem}

\begin{proof}
    First note that if we can find a finite map $g: Y' \rightarrow Y$ such that the theorem holds for $h = g \circ f$, then we are done since if the composition

    $$ \uOmega_{X,\Sigma}^\kdot \rightarrow Rf_* \uOmega_{Y,f^{-1}\Sigma}^\kdot \rightarrow Rh_*\uOmega_{Y',h^{-1}\Sigma}^\kdot $$
    splits then the first map splits. Let $L$ be the Galois closure of the field extension $K(X) \rightarrow K(Y)$. Now we let $Y'$ be the integral closure of $X$ in $L$. Then we have $X = Y' / G$ where $G = \text{Gal}(L / K(X))$ is finite. Now the theorem follows from $\ref{G-splitting}$.
\end{proof}

\begin{corollary} \label{cor-finite-map}
    Let $(X, \Sigma)$ be a reduced pair with $X$ normal and let $f: Y \rightarrow X$ be a surjective finite map. If $(Y, f^{-1} \Sigma)$ is pre-$m$-Du Bois then $(X, \Sigma)$ is pre-$m$-Du Bois.
\end{corollary}

\begin{proof}
    Let $p \le m$. Taking the $p^{th}$ graded piece of the map in \ref{normal-splitting} shows that

    $$ \uOmega_{X,\Sigma}^p \rightarrow f_* \uOmega_{Y,f^{-1}\Sigma}^p $$
    splits. Taking cohomology, we get that

    $$ \tOmega_{X,\Sigma}^p \rightarrow f_* \tOmega_{Y,f^{-1}\Sigma}^p $$
    splits. We have a factorization

    $$ \tOmega_{X,\Sigma}^p \rightarrow \uOmega_{X,\Sigma}^p \rightarrow f_* \uOmega_{Y,f^{-1}\Sigma}^p \cong f_* \tOmega_{Y,f^{-1}\Sigma}^p $$

    Thus the statement follows from \ref{splits}.
\end{proof}


\appendix
\section{Proof of the Injectivity Theorem} \label{proof-inj-thm}

In this appendix, we prove Theorem \ref{inj-thm}. We follow the same strategy as the proof of \cite[Theorem 1.1]{Kov25}. To avoid repetition, we only give the saliant points and refer the reader to \textit{loc cit} for additional details.

Let $\scrL$ be a semiample line bundle on a complex variety $X$. Let $s \in H^0(X,\scrL^n)$ be a general section for some $n >> 0$ and let $H = (s=0)$. Let $0 < i < n$.

\begin{remark}
The technical issue we face is that
\begin{center}
    \begin{tikzcd}
\widetilde{\Omega}^p_{X,\Sigma} \otimes \scrL^{-i} \arrow[r] & \widetilde{\Omega}^p_{X,\Sigma}(\log H) \otimes \scrL^{-i} \arrow[r] & \widetilde{\Omega}^{p-1}_{H,\Sigma_H} \otimes \scrL^{-i}
\end{tikzcd}
\end{center}
is not exact. We would like to be able to just replace $\widetilde{\Omega}^p_{X,\Sigma} \otimes \scrL^{-i}$ with the cone of the right morphism, but this object does not satisfy enough compatibilities to be useful. So instead, we define objects $G^p_{X,\Sigma}(\scrL^{-i})$ and $G^p_{X,\Sigma,H}(\scrL^{-i})$ of $D^b_{coh}(X)$ to replace both the left and middle objects of the above triangle, see \ref{G-dia2}.
    
\end{remark}

Consider the diagram,

\begin{equation} \label{G-def}
\begin{tikzcd}
\widetilde{\Omega}^p_{X,\Sigma}(\log H)[-p] \arrow[r] \arrow[d] & \ftl \arrow[r] \arrow[d] & {\ftl[X][\Sigma][p-1]} \arrow[r, "+1"] \arrow[d] & {} \\
\underline{\Omega}^p_{X,\Sigma}(\log H)[-p] \arrow[r] & \ful \arrow[r] & {\ful[X][\Sigma][p-1]} \arrow[r, "+1"] & {}
\end{tikzcd}
\end{equation}

Define $G^p_{X,\Sigma,H}(\scrL^{-i}) = \textup{Cone}\left( \ftl \otimes \scrL^{-i} \rightarrow {\ful[X][\Sigma][p-1]} \otimes \scrL^{-i} \right)[p-1]$ where the morphism is the composition from \ref{G-def}. We obtain a diagram,

\begin{equation} \label{G-dia1}
\begin{tikzcd}
\widetilde{\Omega}^p_{X,\Sigma}(\log H) \otimes \scrL^{-i} [-p]  \arrow[r] \arrow[d, dashed] & \ftl \otimes \scrL^{-i} \arrow[r] \arrow[d, "id"] & {\ftl[X][\Sigma][p-1]} \otimes \scrL^{-i} \arrow[r, "+1"] \arrow[d] & {} \\
G^p_{X,\Sigma,H}(\scrL^{-i})[-p] \arrow[r] \arrow[d, dashed] & \ftl \otimes \scrL^{-i} \arrow[r] \arrow[d] & {\ful[X][\Sigma][p-1]} \otimes \scrL^{-i} \arrow[r, "+1"] \arrow[d, "id"] & {} \\
\underline{\Omega}^p_{X,\Sigma}(\log H) \otimes \scrL^{-i} [-p] \arrow[r] & \ful \otimes \scrL^{-i} \arrow[r] & {\ful[X][\Sigma][p-1]} \otimes \scrL^{-i} \arrow[r, "+1"] & {}
\end{tikzcd}
\end{equation}

So we can fill in the dashed arrows. Now consider the triangle,

\begin{equation}
\begin{tikzcd}
\underline{\Omega}^p_{X,\Sigma} \otimes \scrL^{-i} \arrow[r] & \underline{\Omega}^p_{X,\Sigma}(\log H) \otimes \scrL^{-i} \arrow[r] & \underline{\Omega}^{p-1}_{H,\Sigma_H} \otimes \scrL^{-i} \arrow[r, "+1"] & {}
\end{tikzcd}
\end{equation}

We define $G^p_{X,\Sigma}(\scrL^{-i}) = \textup{Cone}\left( G^p_{X,\Sigma,H}(\scrL^{-i}) \rightarrow \underline{\Omega}^{p-1}_{H,\Sigma_H} \otimes \scrL^{-i} \right)[-1]$. Thus we have a diagram,

\begin{equation} \label{G-dia2}
\begin{tikzcd}
\widetilde{\Omega}^p_{X,\Sigma} \otimes \scrL^{-i} \arrow[d, dashed] \arrow[r] & \widetilde{\Omega}^p_{X,\Sigma}(\log H) \otimes \scrL^{-i} \arrow[r] \arrow[d] & \widetilde{\Omega}^{p-1}_{H,\Sigma_H} \otimes \scrL^{-i} \arrow[d] & {} \\
G^p_{X,\Sigma}(\scrL^{-i}) \arrow[r] \arrow[d, dashed] & G^p_{X,\Sigma,H}(\scrL^{-i}) \arrow[r] \arrow[d] & \underline{\Omega}^{p-1}_{H,\Sigma_H} \otimes \scrL^{-i} \arrow[r, "+1"] \arrow[d, "id"] & {} \\
\underline{\Omega}^p_{X,\Sigma} \otimes \scrL^{-i} \arrow[r] & \underline{\Omega}^p_{X,\Sigma}(\log H) \otimes \scrL^{-i} \arrow[r] & \underline{\Omega}^{p-1}_{H,\Sigma_H} \otimes \scrL^{-i} \arrow[r, "+1"] & {}
\end{tikzcd}
\end{equation}
We can fill in the top dashed arrow since the composition $\widetilde{\Omega}^p_{X,\Sigma}\to \underline{\Omega}^{p-1}_{H,\Sigma_H} \otimes \scrL^{-i}$ is $0$. We can once again fill in the bottom dashed arrows since the bottom two rows are exact.

\begin{lemma} \label{m-1-iso}
    If $(X, \Sigma)$ is pre-$(m-1)$-Du Bois then $G^p_{X,\Sigma,H}(\scrL^{-i}) \cong \widetilde{\Omega}^p_{X,\Sigma}(\log H) \otimes \scrL^{-i}$ and $G^p_{X,\Sigma}(\scrL^{-i})|_{X \setminus H} \cong \underline{\Omega}^p_{X,\Sigma} \otimes \scrL^{-i}|_{X \setminus H}$ for all $p \le m$.
\end{lemma}
\begin{proof}
    This follows by the top two rows of \ref{G-dia1} and \ref{G-dia2} along with Lemma \ref{lem-premDB-hyperplane}.
\end{proof}

\begin{theorem} \label{G-surj}
Let $X$ be a proper complex variety with a semiample line bundle $\scrL$. For $0 < i < n$, the following maps are surjective:

\begin{enumerate}
    \item $\HH^q(X, G^p_{X,\Sigma,H}(\scrL^{-i})) \rightarrow \HH^q(X, \underline{\Omega}^p_{X,\Sigma}(\log H) \otimes \scrL^{-i})$
    \item $\HH^q(X, G^p_{X,\Sigma}(\scrL^{-i})) \rightarrow \HH^q(X, \underline{\Omega}^p_{X,\Sigma} \otimes \scrL^{-i})$
\end{enumerate}
\end{theorem}

\begin{proof}
    The bottom two rows of \ref{G-dia1} induce 
\vspace{4mm}

\adjustbox{scale=0.80,center}{%
\begin{tikzcd}[column sep=tiny]
\HH^{q+p-1}(X,{\ful[X][\Sigma][p-1]} \otimes \scrL^{-i}) \arrow[r] \arrow[d, "id"] & \HH^q(X,G^p_{X,\Sigma,H}(\scrL^{-i})) \arrow[r] \arrow[d] & \HH^{q+p}(X, \ftl \otimes \scrL^{-i}) \arrow[r] \arrow[d] & \HH^{q+p}(X,{\ful[X][\Sigma][p-1]} \otimes \scrL^{-i}) \arrow[d, "id"] \\
\HH^{q+p-1}(X,{\ful[X][\Sigma][p-1]} \otimes \scrL^{-i}) \arrow[r] & \HH^{q}(X,\underline{\Omega}^p_{X,\Sigma}(\log H) \otimes \scrL^{-i}) \arrow[r] & \HH^{q+p}(X,\ful \otimes \scrL^{-i}) \arrow[r] & \HH^{q+p}(X,{\ful[X][\Sigma][p-1]} \otimes \scrL^{-i})
\end{tikzcd}
} 
\vspace{4mm}
    (1) now follows by \ref{prop_surjecitivty_cover} and the 4-lemma. 

    The bottom two rows of \ref{G-dia2} induce
\vspace{4mm}

\adjustbox{scale=0.80,center}{%
\begin{tikzcd}[column sep=tiny]
\HH^{q-1}(X,\underline{\Omega}^{p-1}_{H,\Sigma_H} \otimes \scrL^{-i}) \arrow[r] \arrow[d, "id"] & \HH^q(X,G^p_{X,\Sigma}(\scrL^{-i})) \arrow[r] \arrow[d] & \HH^{q}(X, G^p_{X,\Sigma, H}(\scrL^{-i})) \arrow[r] \arrow[d] & \HH^{q}(X,\underline{\Omega}^{p-1}_{H,\Sigma_H} \otimes \scrL^{-i}) \arrow[d, "id"] \\
\HH^{q-1}(X,\underline{\Omega}^{p-1}_{H,\Sigma_H} \otimes \scrL^{-i}) \arrow[r] & \HH^{q}(X,\underline{\Omega}^p_{X,\Sigma} \otimes \scrL^{-i}) \arrow[r] & \HH^{q}(X,\underline{\Omega}^p_{X,\Sigma}(\log H) \otimes \scrL^{-i}) \arrow[r] & \HH^{q}(X,\underline{\Omega}^{p-1}_{H,\Sigma_H} \otimes \scrL^{-i})
\end{tikzcd}
} \vspace{4mm}
    Now part (1) and the 4-lemma imply (2).
\end{proof}

\begin{proof}[Proof of Theorem \ref{inj-thm}]

    From here the proof follows exactly the same as the proof of \cite[Theorem 9.1]{Kov25}. To avoid repetition, we will do the simpler case where we assume $X$ is projective. We hope this simpler case can serve as a guide for readers interested in the proof of the general case.

    Let $\mathscr{K}$ be the kernel of the map

    $$ h^i( \bbD_X(\ou)) \rightarrow h^i (\bbD_X(\ot)) $$
    Let $Z = \textrm{Supp } \mathscr{K}$. The first step of the proof is a hyperplane cut down argument to show that $\dim Z = 0$. This part is the same as in \cite[Theorem 9.1]{Kov25} so we omit it.

    Now choose a very ample line bundle $\scrL$ such that
    \begin{equation} \label{coh-vanishing}
        H^j(X, h^i(\bbD_X(\ou)) \otimes \scrL) = 0
    \end{equation}
    for all $i$ and $j > 0$. Now choose $n >> 0$ and a general section $s \in H^0(X, \scrL^n)$ to define $H$ and $\ugt$ as above. Let $t \ne 0 \in H^0(X, \mathscr{K} \otimes \scrL)$ be a section. We may also consider $t$ as a element of $H^0(X, h^i( \bbD_X(\ou) \otimes \scrL))$. Now consider the map
    $$ \mathbb{H}^i(X, \bbD_X(\ou) \otimes \scrL) \hookrightarrow \mathbb{H}^i(X, \bbD_X(\ugt) \otimes \scrL) $$
    which is injective by Grothendieck duality applied to the map in Theorem \ref{G-surj} (2).
    The left side is equal to $H^0(X, h^i( \bbD_X(\ou) \otimes \scrL))$ by \ref{coh-vanishing}. Now we apply \cite[Proposition 2.3, \href{https://arxiv.org/abs/2505.09912v3}{version 3}]{Kov25} to get that the map

    $$ H^0(X, h^i(\bbD_X(\ou) \otimes \scrL)) \overset{\gamma}{\hookrightarrow} H^0(X, h^i(\bbD_X(\ugt) \otimes \scrL)) $$
    is injective. We have that $\ugt{|_{X \setminus H}} \cong \ot{|_{X \setminus H}}$ by Lemma \ref{m-1-iso}. Therefore $\gamma(t)_{|V} = 0$ by the definition of $\mathscr{K}$. Thus $\gamma(t)$ is supported on $H$. Since $t$ is supported on $Z$, $\gamma(t)$ is supported on $Z$. Since $Z$ is finite and $H$ is general, $Z \cap H = \varnothing$. Therefore $\gamma(t) = 0$ contradicting the injectivity of $\gamma$.
\end{proof}

\bibliographystyle{alpha} 
\bibliography{ref}

\end{document}